\documentclass[twoside,11pt]{article}

\usepackage{jmlr2e}
\usepackage{amsmath}
\usepackage{amssymb}
\usepackage{graphicx}
\usepackage{color}
\usepackage{mathdots}
\usepackage{mathtools}
\usepackage{bbm}
\usepackage{subfigure}
\usepackage{lastpage}

\DeclareMathOperator*{\argmax}{argmax}
\DeclareMathOperator*{\argmin}{argmin}

\DeclareMathOperator*{\rank}{rank}
\DeclareMathOperator*{\cone}{cone}

\newcommand{\reals}{\mathbb{R}}
\newcommand{\Sn}{{\mathcal S}(n)}
\newcommand{\Ln}{{\mathcal L}_n}

\newcommand{\Lnk}{{\mathcal L}_{n,k}}
\newcommand{\Qnk}{{\mathcal Q}_{n,k}}

\jmlrheading{23}{2022}{1-\pageref{LastPage}}{4/21; Revised
1/22}{6/22}{21-0402}{Pedro Felzenszwalb, Caroline Klivans and Alice Paul}
\ShortHeadings{Clustering with Semidefinite Programming}{Felzenszwalb, Klivans and Paul}

\firstpageno{1}

\begin{document}

\title{Clustering with Semidefinite Programming and \\ Fixed Point Iteration}
\author{\name Pedro Felzenszwalb \email pff@brown.edu \\
  \addr School of Engineering and Department of Computer Science \\
  Brown University \\
  Providence, RI 02912, USA
  \AND
  \name Caroline Klivans \email klivans@brown.edu \\
  \addr Division of Applied Mathematics \\
  Brown University \\
  Providence, RI 02912, USA
  \AND Alice Paul \email alice\_paul@brown.edu \\
  \addr Department of Biostatistics \\
  Brown University \\
  Providence, RI 02912, USA}

\editor{Silvia Villa}
\maketitle

\begin{abstract}%
    We introduce a novel method for clustering using a semidefinite
    programming (SDP) relaxation of the Max $k$-Cut problem.  The
    approach is based on a new methodology for rounding the solution
    of an SDP relaxation using iterated linear optimization.  We show
    the vertices of the Max $k$-Cut relaxation correspond to
    partitions of the data into at most $k$ sets.  We also show the
    vertices are attractive fixed points of iterated linear
    optimization.  Each step of this iterative process solves a
    relaxation of the closest vertex problem and leads to a new
    clustering problem where the underlying clusters are more clearly
    defined.  Our experiments show that using fixed point iteration
    for rounding the Max $k$-Cut SDP relaxation leads to significantly
    better results when compared to randomized rounding.
\end{abstract}


\section{Introduction}
\label{sec:intro}

Semidefinite programming (SDP) relaxations have led to significant
advances in the development of combinatorial optimizaton algorithms.
This includes a variety of problems with applications to machine
learning.  Many challenging optimization problems can be approximately
solved by a combination of an SDP relaxation and a rounding step.  One
of the best examples of this paradigm is the celebrated Max Cut
approximation algorithm of \cite{Goemans}.

From a theoretical point of view algorithms based on SDP relaxations
can lead to strong approximation guarantees.  However, such
approximation guarantees do not always translate to practical
solutions.  Many algorithms with good theoretical guarantees rely on
randomized rounding methods that can produce solutions that have
undesirable artifacts despite having high objective value.
This motivates the development of effective deterministic methods for
rounding the solutions of SDP relaxations.  Recent advances based on
the sum-of-squares hierarchy have also motivated the development of
new general methods for rounding the solutions of SDP relaxations
(see, e.g., \cite{Barak}).

In this paper we introduce a novel method for clustering using the Max
$k$-Cut SDP relaxation described in \cite{Frieze}.  The approach is
based on a new methodology for rounding the solution of an SDP
relaxation introduced by the current authors in \cite{iter}.  Our
rounding method involves fixed point iteration with a map that
optimizes a linear function over a convex body.
Figure~\ref{fig:example} shows a clustering example, comparing the
result of our fixed point iteration method for rounding the solution
of the Max $k$-Cut relaxation to the result obtained using the
randomized rounding method in \cite{Frieze}.

\begin{figure}
\centering
  \subfigure[Input data]{\includegraphics[width=2.8in]{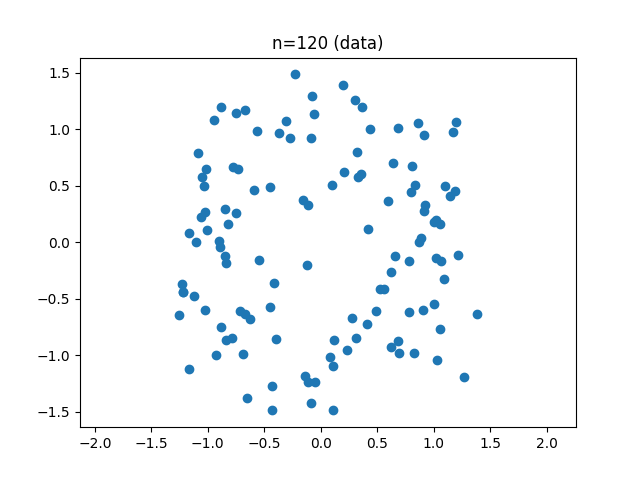}}
  
  \subfigure[Randomized rounding (best of 50 trials)]{\includegraphics[width=2.8in]{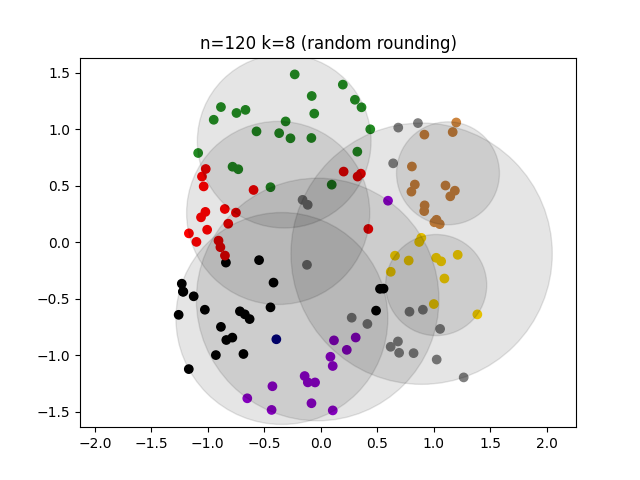}}
  \subfigure[Fixed point iteration]{\includegraphics[width=2.8in]{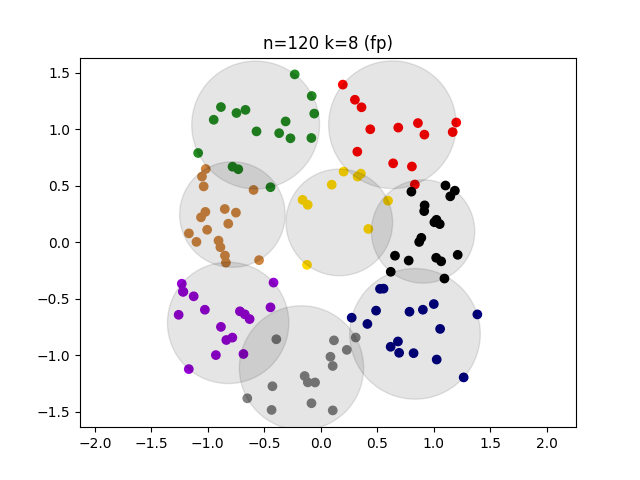}}
  
  \subfigure[Sequence of solutions generated by fixed point iteration]{\includegraphics[trim=0 130 45 130, clip, width=7in]{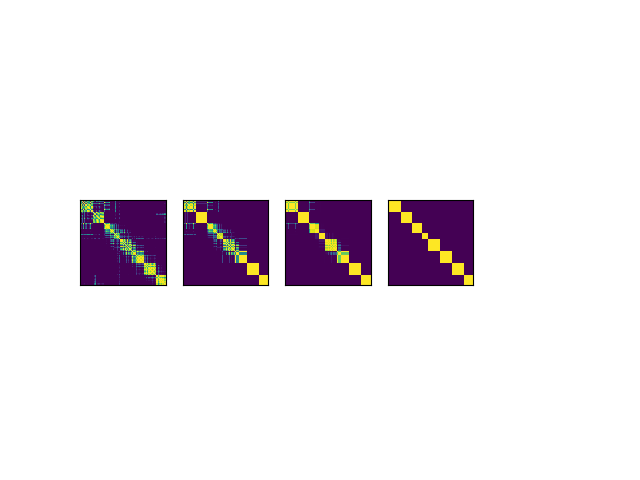}}
  
  \caption{Clustering 120 points into 8 clusters using the Max $k$-Cut SDP relaxation: (a) input data, (b)
    clustering using randomized rounding, and (c) clustering using
    fixed point iteration. Each cluster is shown with a different
    color and a minimal enclosing circle.  A solution to the Max
    $k$-Cut SDP relaxation is a matrix in the $k$-way elliptope.  (d)
    illustrates the sequence of matrices obtained by fixed point
    iteration.  The final matrix is an
    integer solution defining a clustering.  }
\label{fig:example}
\end{figure}

The SDP relaxation for Max $k$-Cut involves linear optimization over a
convex body that we call the $k$-way elliptope.  In
Section~\ref{sec:kway} we show that the vertices of the $k$-way
elliptope correspond to partitions (clusterings) of the data into at
most $k$ sets, generalizing the result from \cite{Laurent} for the
elliptope (the $k=2$ case).

As pointed out already by the authors of \cite{Frieze}, the randomized
rounding method for the Max $k$-Cut SDP relaxation
has some significant shortcomings.  The approximation
factor of the randomized algorithm appears good on the surface, but is
not much better than the approximation factor one gets by simply randomly
partitioning the data.  Randomized rounding often generates a
partition with fewer than $k$ sets.  The result in
Figure~\ref{fig:example}(b) is a partition with 7 non-empty clusters
despite the fact that $k=8$.  We also see in
Figure~\ref{fig:example}(b) that the resulting clusters are not
compact.  Instead different clusters have significant overlap.  In
Section~\ref{sec:experiments} we compare our fixed point iteration
method to the randomized rounding procedure in several examples,
showing that the fixed point approach can produce much better
clusterings in practice.   

The work in \cite{Mahajan} gives a general method for derandomizing
approximation algorithms based on SDP relaxations, including the Max
$k$-Cut relaxation we use for clustering.  Although the method in
\cite{Mahajan} is interesting from a theoretical point of view, the
approach is not practical for problems of non-trivial size.  More
recent methods for derandomizing the Max Cut approximation algorithm
include \cite{odonnell} and \cite{bhargava}.  These methods are all
based on the randomized rounding method of \cite{Goemans} but replace
the randomization with a search over a limited number of discretized
choices.  In contrast, our fixed point iteration method is not based
on derandomization techniques and is instead based on a novel approach
for rounding the solution of an SDP relaxation.

Intuitively the problem of rounding a solution of the SDP relaxation
for Max $k$-Cut can be interpreted as a new clustering problem. This
motivates the iterative nature of our algorithm.  Our rounding
procedure solves a sequence of SDP problems where the underlying
clusters become more clearly defined in each iteration.  In
\cite{iter} we showed that iterated linear optimization in a convex
region always converges to a fixed point.  In the case of the $k$-way
elliptope the integer solutions to the Max $k$-Cut problem are
attractive fixed points.  Additionally, when rounding the solution of
the SDP relaxation, we ideally round to the closest vertex ($k$-way
partition).  We show in Section~\ref{sec:iter} that each iteration of
our rounding procedure corresponds to a relaxation of this objective.

\cite{mixon2017clustering} also solve an SDP relaxation to cluster
subgaussian mixtures. However, rather than round this SDP solution
directly, they revert back to employing Lloyd's algorithm to partition
the columns of the solution. Our method, instead, continues to use the
strength of semidefinite programming to recover a partition.

For some applications convex relaxations have been shown to recover
the ``true'' hidden structure in the data (see, e.g., \cite{Candes}).
For clustering applications it was shown in \cite{Ward} and
\cite{mixon2017clustering} that convex relaxations can recover a
ground truth clustering if the data is sufficiently well-separated.
However, in practice the data is rarely well-separated (and there is
often no ground truth clustering).  Nonetheless, good clusterings
might exist that can be extremely useful for data processing, coding
or analysis.  The data in Figure~\ref{fig:example} illustrates an
example of this situation.  The data was generated by sampling from
several Gaussian distributions with significant overlap.  In this case
there is no way to recover the ground truth clustering, but the result
of our fixed point iteration method still provides a good solution
that can be used for subsequent analysis.

In Section~\ref{sec:clustering} we discuss how Max $k$-Cut can be used
to formulate the clustering problem.  In Section~\ref{sec:sdp} we
review the Max $k$-Cut SDP relaxation and the randomized rounding
method from \cite{Frieze}.  In Section~\ref{sec:kway} we study the
convex body that arises from the SDP relaxation and show the vertices
of the feasible region correspond to partitions.  In
Section~\ref{sec:iter} we describe how iterated linear optimization
leads to a deterministic method for rounding the solution of the SDP
relaxation.  The approach solves a sequence of relaxations to the
closest vertex problem, and leads to clusters that are more clearly
defined in each iteration.  Finally, in Section~\ref{sec:experiments}
we illustrate experimental results of our new rounding method and
compare them to the randomized rounding approach from \cite{Frieze}.
We also compare the result of our method to the $k$-means algorithm
for clustering images in the MNIST handwritten digit dataset.

\section{Clustering with Max $k$-Cut}
\label{sec:clustering}

Clustering problems are often formulated using pairwise measures of
similarity or dissimilarity between objects.  Intuitively we would
like to partition the data so that pairs of objects within a cluster
are similar to each other while pairs of objects in different clusters
are dissimilar.  This natural idea leads to a variety of formulations
of clustering as graph partition problems. One benefit of graph-based
approaches to clustering is the ability to define pairwise measures
that incorporate both categorical and numerical variables (see,
e.g.\ \cite{bertsimas2021interpretable}).

Let $G=(V,E)$ be a weighted graph.  Let $[n] = \{1,\ldots,n\}$.  To
simplify notation we assume throughout that $V=[n]$ and that the graph
is complete.  A $k$-partition is a partition of $V$ into $k$ disjoint
sets $(A_1,\ldots,A_k)$, some of which may be empty.  Let $M$ be a
symmetric matrix of pairwise non-negative weights.  The weight of a
partition $P$ is the sum of the weights of the pairs $\{i,j\}
\subseteq [n]$ that are split (or cut) by $P$,
$$w(P) = \sum_{1 \le r < s \le k}\;\; \sum_{i \in A_r}\; \sum_{j \in A_s} M_{i,j}.$$
The Max $k$-Cut problem is to find a $k$-partition maximizing $w(P)$.  

As an example, let $D=\{x_1,\ldots,x_n\}$ be $n$ points in $\reals^d$.
One of the most commonly used formulations for clustering data in
Euclidean space involves optimizing the $k$-means objective.  For $A
\subseteq [n]$ let $m(A)$ be the mean of the points indexed by $A$,
$$m(A) = \frac{1}{|A|} \sum_{i \in A} x_i.$$ The $k$-means objective
(\ref{eqn:kmeans}) is to partition the data into clusters to minimize
the sum of squared distances from each point to the center of its
cluster,
\begin{equation}
\label{eqn:kmeans}
\argmin_{(A_1,\ldots,A_k)} \; \sum_{r=1}^k \; \sum_{i \in A_r} ||x_i-m(A_i)||^2.
\end{equation}
This objective encourages partitions of the data into ``compact''
clusters, and is widely used both for clustering and for vector
quantization in coding applications (see, e.g.,
\cite{Lloyd,Mackay,Jain,kmeanspp}).

Let $M_{i,j} = ||x_i-x_j||^2$ and consider the Max $k$-Cut objective
with these weights.  Maximizing the weight of the pairs
$\{i,j\}$ that are split by a partition is the same as minimizing the
weights of the pairs $\{i,j\}$ that are not split.  Moreover, the sum
of squared distances between points within a set $A_i$ can be
expressed in terms of the sum of squared distances between each point
in $A_i$ and the mean of the set:
\begin{align}
\argmax w(A_1,\ldots,A_k) & = \argmin_{(A_1,\ldots,A_k)} \; \sum_{r=1}^k \; \sum_{i,j \in A_r} ||x_i-x_j||^2,\\
                          & = \argmin_{(A_1,\ldots,A_k)} \; \sum_{r=1}^k \; |A_r| \sum_{i \in A_r} ||x_i-m(A_r)||^2.
                          \label{eqn:maxcutEuclidean}
\end{align}
The objective function (\ref{eqn:maxcutEuclidean}) is similar to the
$k$-means objective (\ref{eqn:kmeans}), except that in the case of Max
$k$-Cut there is a preference towards balanced partitions.  In
Section~\ref{sec:experiments} we illustrate the results of clustering
experiments with Max $k$-Cut using this formulation.

The work in \cite{wang} also considered clustering with Max $k$-Cut
and used an SDP relaxation to solve the resulting problem.  That work
focused on an information theoretical formulation of the clustering
problem that could also be used within our framework.  SDP relaxations
of the $k$-means objective for clustering have also been considered in
\cite{Peng}, \cite{Ward}, and \cite{mixon2017clustering}.

Most of the previous work on clustering using graph-based methods has
focused on formulations based on minimum cuts and spectral algorithms
(\cite{Shi,Meila,Ng,Weiss,Kannan}). In this case the weight of an edge
represents similarity (or affinity) between elements instead of
dissimilarity.  Minimum cut formulations often include some form of
normalization (see, e.g., \cite{Shi,Kannan}) or a balance requirement
to avoid trivial partitions (otherwise the cut is minimized when one
partition is very small).  In contrast to minimum cut formulations,
clustering using a maximum cut formulation naturally encourages
balanced partitions, as they maximize the total number of edges that
are split.

\section{SDP Relaxation for Max $k$-Cut}
\label{sec:sdp}

The \cite{Goemans} approximation algorithm for Max Cut
(clustering into \emph{two} clusters) is based on an SDP relaxation
and a randomized rounding method.  The relaxation
involves the optimization of a linear function over a convex body
$\Ln$ known as the \emph{elliptope}.
\cite{Laurent} showed that the vertices of $\Ln$ correspond to
bipartitions of $[n]$.  The fact that the vertices of $\Ln$ are
bipartitions gives an explanation as to why in some cases the SDP
relaxation can lead directly to an integer solution that is an
\emph{optimal} solution to the Max Cut problem (see also
\cite{Cifuentes}).

\cite{Frieze} generalized the approximation algorithm of Goemans and
Williamson to a randomized approximation method for Max $k$-Cut.  In
this case the algorithm involves linear optimization over a convex
body $\Lnk$ that we call the \emph{$k$-way elliptope}.

Below we briefly review the SDP relaxation and randomized rounding
method for Max $k$-Cut introduced by \cite{Frieze}.  In
Section~\ref{sec:kway} we show the vertices of the $k$-way elliptope
correspond to $k$-partitions, generalizing the result from
\cite{Laurent} for the elliptope.  In Section~\ref{sec:iter} we
describe a determinisic method for rounding the solution of the SDP
relaxation for Max $k$-Cut based on iterated linear optimzation.

The Max $k$-Cut SDP relaxation is based on a reformulation of the
combinatorial problem in terms of Gram matrices.  Let $a_1,\ldots,a_k$
be the vertices of an equilateral simplex $\Sigma_k$ in $\reals^{k-1}$
centered around the origin and scaled such that $||a_i|| = 1$.  

A $k$-partition $P=(A_1,\ldots,A_k)$ can be encoded by $n$ vectors
$(y_1,\ldots,y_n)$ with $y_i = a_j$ if $i \in A_j$.  Define the
$k$-partition matrix $X^P$ to be the Gram matrix of
$(y_1,\ldots,y_n)$.  For $i
\neq j$ we have $a_i \cdot a_j = -1/(k-1)$.  Therefore,
\begin{align}
X^P_{i,j} & = 
\begin{cases}
   1 & \text{if $\{i,j\}$ are together in $P$} \\
   -1/(k-1) & \text{if $\{i,j\}$ are split by $P$}
\end{cases} \\
1-X^P_{i,j} & =
\begin{cases}
   0 & \text{if $\{i,j\}$ are together in $P$} \\
   k/(k-1)\phantom{-} & \text{if $\{i,j\}$ are split by $P$}
\end{cases}
\end{align}

For two $n$ by $n$ matrices $X$ and $Y$ let
$$X \cdot Y = \sum_{i=1}^{n} \sum_{j=1}^{n} X_{i,j} Y_{i,j}.$$

Now the weight of a partition, $w(P)$, can be written as,
$$w(P) = \frac{k-1}{2k} (1-X^P) \cdot M,$$
where $M$ is the symmetric matrix of pairwise weights.

\begin{definition}
The set of \emph{$k$-partition matrices} $\Qnk$ is the set of Gram
matrices of $n$ vectors $(y_1,\ldots,y_n)$ with $y_i \in
\{a_1,\ldots,a_k\}$ for $1 \le i \le n$.
\end{definition}

We can reformulate the Max $k$-Cut problem as an optimization over
$k$-partiton matrices,
$$\argmax_{X \in \Qnk} \frac{k-1}{2k} (1-X) \cdot M.$$ The SDP
relaxation of Max $k$-Cut is based on relaxing the requirement that
$y_i \in \{a_1,\ldots,a_n\}$ in the definition of $\Qnk$ to allow
$y_i$ to be any unit vector in $\reals^n$, with the additional
constraint that $y_i \cdot y_j \ge -1/(k-1)$.

Let $\Sn \subset \reals^{n \times n}$ be the set of $n \times n$
symmetric matrices.

\begin{definition}
The \emph{elliptope} $\Ln$ is the subset of matrices in $\Sn$ that are
positive semidefinite and have all 1’s on the diagonal:
$$\Ln = \{ X \in \Sn \,|\, X \succcurlyeq 0, X_{i,i} =
1 \}.$$
\end{definition}
The matrices in $\Ln$ exactly correspond to Gram matrices of $n$ unit
vectors $(y_1,\ldots,y_n)$ in $\reals^n$ (\cite{Goemans,Laurent}).

\begin{definition}
The \emph{$k$-way elliptope} $\Lnk$ is the subset of
matrices in $\Ln$ where every entry is at least $-1/(k-1)$:  
$$\Lnk = \{ X \in \Ln \,|\, X_{i,j} \ge -1/(k-1) \}.$$
\end{definition}
The matrices in $\Lnk$ correspond to Gram matrices of $n$ unit vectors
$(y_1,\ldots,y_n)$ in $\reals^n$ with $y_i \cdot y_j \ge -1/(k-1)$.

The randomized algorithm in \cite{Frieze} involves the SDP relaxation,
$$\argmax_{X \in \Lnk} \frac{k-1}{2k} (1-X) \cdot M.$$

Let $X$ be an optimal solution to the SDP relaxation.  We can interpret $X$
as the Gram matrix of $n$ unit vectors $(y_1,\ldots,y_n)$ in
$\reals^n$, obtained using a Cholesky decomposition $X=VV^T$.  To
generate a $k$-partition the rounding method in \cite{Frieze}
selects $k$ unit
vectors $(u_1,\ldots,u_k)$ independently from a uniform
distribution and assigns $y_i$ to the closest vector $u_j$.

When $k>2$ rounding a $k$-partition matrix $X^P$ using the randomized
procedure above can generate a partition that is different from $P$
because different sets in $P$ can be merged.  More generally 
randomized rounding often generates a partition with fewer than
$k$ sets because some vector $u_j$ is not the closest vector to any
$v_i$.  In Section~\ref{sec:experiments} we show experimental results
that illustrate several problems that arise in practice when using the
randomized rounding procedure.  Even when the cut value of the random
partition is relatively high, the resulting clustering can have
undesirable artifacts.

\section{The $k$-way elliptope}
\label{sec:kway}

The main result of this section is that the vertices of $\Lnk$ are the
matrices in $\Qnk$ and correspond to $k$-partitions of $[n]$.  Note
that a $k$-partition may have some empty sets, so the vertices of
$\Lnk$ correspond to partitions of $[n]$ into at most $k$ sets.

For a matrix $X \in {\mathcal L}_{3,k}$ we have
$$X =
\begin{pmatrix}
1 & x & y \\
x & 1 & z \\
y & z & 1
\end{pmatrix}.$$
Therefore we can visualize $X$ as a point $(x,y,z) \in \mathbb{R}^3$.      

\begin{figure}
  \centering
  \vspace{-1cm}
  \includegraphics[width=4in]{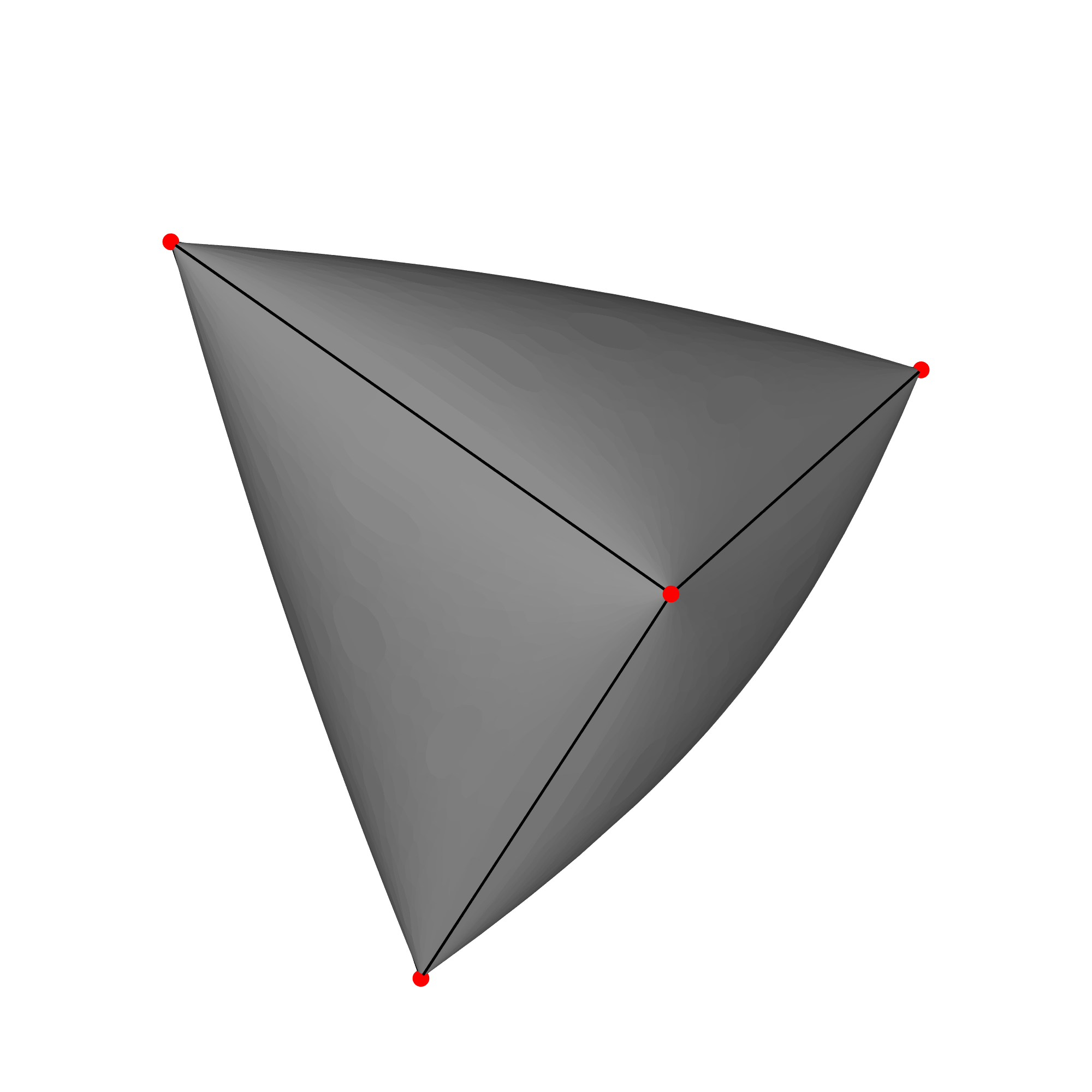}
  \caption{The $2$-way elliptope ${\cal L}_{3,2}$ has 4 vertices (red points),
  corresponding to partitions of 3 distinguished elements into at most 2 sets.} 
    \label{fig:elliptope}
\end{figure}

\begin{figure}
  \centering
  \vspace{-1cm}
  \includegraphics[width=4in]{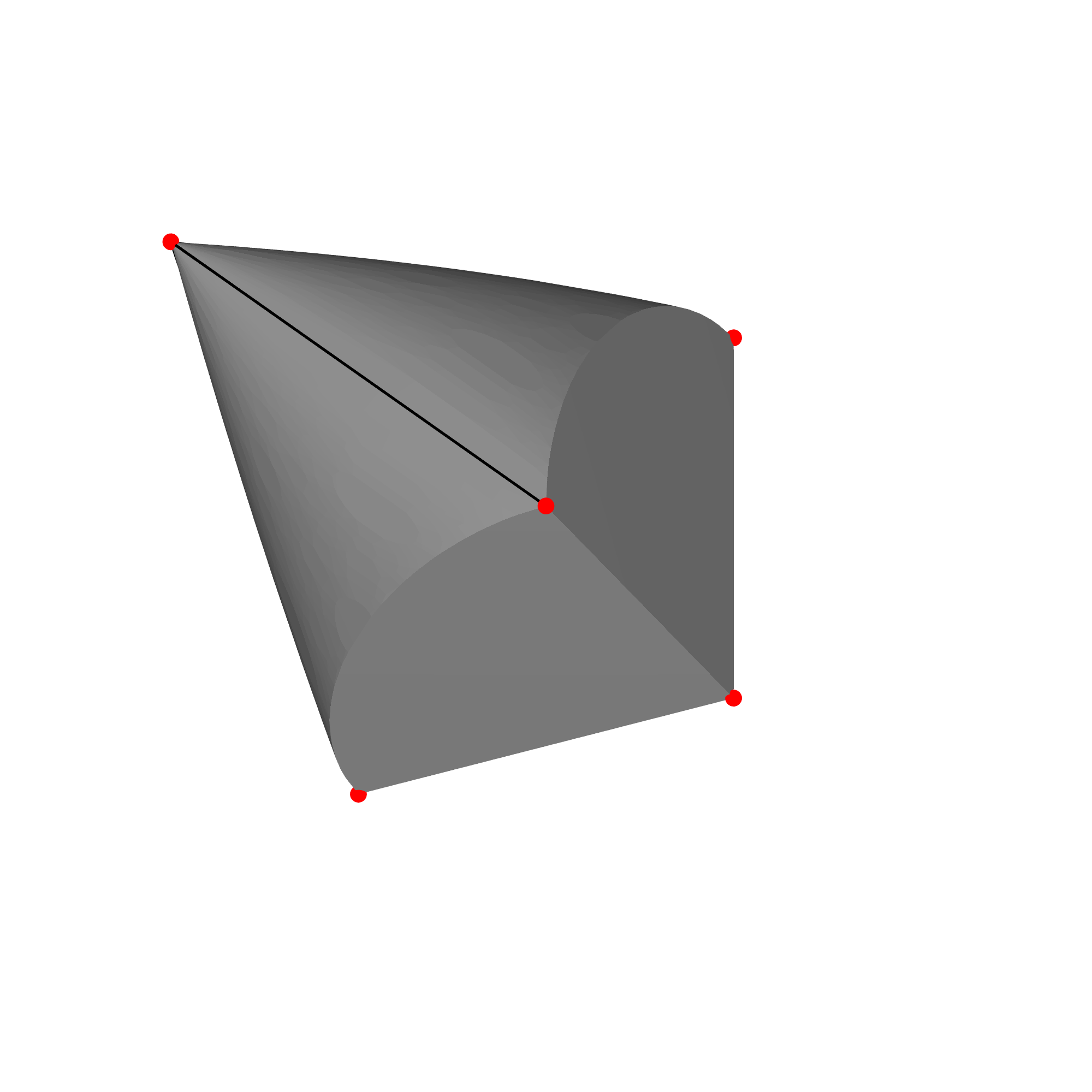}
  \vspace{-1.4cm}
  \caption{The $3$-way elliptope ${\cal L}_{3,3}$ has 5 vertices (red points),
    corresponding to partitions of 3 distinguished elements into
    at most 3 sets.}
  \label{fig:kway}
\end{figure}

Figure~\ref{fig:elliptope} illustrates ${\cal L}_{3,2}$.  
This convex body has 4 vertices, with one vertex for each partition of 3
distinguished elements into at most 2 sets.  The partitions are listed below.

$P_1 = (\{1,2,3\},\emptyset)$

$P_2 = (\{1,2\},\{3\})$

$P_3 = (\{1,3\},\{2\})$

$P_4 = (\{2,3\},\{1\})$

Figure~\ref{fig:kway} illustrates ${\cal L}_{3,3}$.  
This convex body has 5 vertices, with one vertex for each 
partition of 3 distinguished
elements into at most 3 sets.  The partitions are listed below.

$P_1 = (\{1,2,3\},\emptyset,\emptyset)$

$P_2 = (\{1,2\},\{3\},\emptyset)$

$P_3 = (\{1,3\},\{2\},\emptyset)$

$P_4 = (\{2,3\},\{1\},\emptyset)$

$P_5 = (\{1\},\{2\},\{3\})$

Note that $\Lnk$ is simply the elliptope intersected with an orthant.
This characterization is useful to understand the geometric and
combinatorial structure of $\Lnk$.  The difference between ${\cal
  L}_{n,r}$ and ${\cal L}_{n,s}$ is the amount by which the
intersecting orthant is translated.

If we translate the orthant continuously from $0$ to $-1/(n-1)$ the
vertex that represents the grouping of all elements into a single set
remains fixed.  This vertex is the matrix of all 1's.  On the other
hand, the vertex that represents the partition of all elements into
different sets only appears when the orthant reaches $-1/(n-1)$.  This
vertex is the $n$-partition matrix with 1's on the diagonal and
$-1/(n-1)$ in the off diagonal entries.

Let $\Delta$ be a convex subset of $\reals^n$.  For $x \in
\Delta$, the \emph{normal cone of $\Delta$ at $x$} is the set
$$N(\Delta,x) = \{ y \in \reals^n \,|\, y \cdot x \ge y \cdot z \;\; \forall z \in \Delta \}.$$

A vertex of $\Delta$ is (by definition) a point with a full-dimensional normal cone.

The following result shows that the normal cone of $\Lnk$ at a $k$-partition
matrix is full-dimensional.

\begin{proposition}
  If $X \in \Qnk$ and $||Y-X|| < 1/(k-1)$ then $Y \in N(\Lnk,X)$.
  \label{prop:attract}
\end{proposition}
\begin{proof}
  If $||Y-X|| < 1/(k-1)$ then
  $|(Y-X)_{i,j}| < 1/(k-1)$ for all $\{i,j\} \subseteq [n]$.  
  Since $X_{i,j}$ is
  either $1$ or $-1/(k-1)$ we can see that $Y$ has the same
  sign pattern as $X$.
  If $Z \in \Lnk$ all of the entries in $Z$ are between $-1/(k-1)$ and
  $1$.  Therefore $Y \cdot X \ge Y \cdot Z$.
\end{proof}

The next proposition shows that $k$-partition matrices are the only
matrices in $\Lnk$ where every entry is either $-1/(k-1)$ or $1$.

\begin{proposition}
  If $X \in \Lnk$ and $X_{i,j} \in \{ -1/(k-1),1\}$ for all $\{i,j\}
  \subseteq [n]$ then $X \in \Qnk$.
  \label{prop:entry}
\end{proposition}
\begin{proof}
Suppose $X_{i,j} \in \{-1/(k-1),1\}$ for all $\{i,j\} \subseteq [n]$.
Since $X$ is a Gram matrix we know that $X_{i,j} = 1$ defines an equivalence 
relation on $[n]$.  Let $(A_1,\ldots,A_l)$ be the equivalence classes of 
this relationship.   If $i \in A_r$ and $j \in A_s$ with $r \neq s$ then
$X_{i,j} = -1/(k-1)$.  This means we have $l$ unit vectors with the dot product
between each pair equal to $-1/(k-1)$.  
Lemma 4 in \cite{Frieze} implies that $l \le k$.  We conclude $X \in \Qnk$.
\end{proof}

Let $$T(M) = \argmax_{X \in \Ln} M \cdot X.$$  
Note that the $\argmax$ in the definition of $T$ may not be unique.  In
this case $T(M)$ is set valued.  When we write $X = T(M)$ we allow $X$ to
be \emph{any} element of $T(M)$.

We will use the following
lemma from \cite{iter}.

\begin{lemma}
  \label{lem:local}
  Let $M \in \Sn$ and $X=T(M)$.  
  
  Suppose $X$ is the Gram matrix of $n$ unit vectors $(v_1,\ldots,v_n)$.  Then,
  \begin{itemize}
  \item[(a)]There exists real values $\alpha_i$ such that 
  $$\sum_{j \neq i} M_{i,j} v_j = \alpha_i v_i.$$
  \item[(b)]The vectors $(v_1,\ldots,v_n)$ are linearly dependent and $\rank(X) < n$.
  \item[(c)] There exists a diagonal matrix $D$ such that,
    $$MX = DX.$$
  \end{itemize}
\end{lemma}

Now we are ready to prove the main result of this section, showing
that the vertices of $\Lnk$ are in fact the integer solutions to the
Max $k$-Cut problem.

\begin{theorem}
The vertices of $\Lnk$ are the $k$-partition matrices.
\end{theorem}
\begin{proof}
Suppose $X \in \Qnk$.  Proposition~\ref{prop:attract} implies 
$N(\Lnk,X)$ is full-dimensional.

Now suppose $X \not\in \Qnk$.  Proposition~\ref{prop:entry} implies
that $X$ must have at least one entry $X_{i,j} \not\in \{-1/(k-1),1\}$ Since $\Lnk
\subseteq \Ln$ we know $X \in \Ln$.  Note that $X_{i,j} \not\in
\{-1,1\}$.  Let $J$ be a matrix with $J_{i,j} = J_{j,i} = 1$ and
$J_{k,l} = 0$ for $\{k,l\} \neq \{i,j\}$.

Let ${\cal O}_{n,k} = \{X \,|\, X_{i,j} \ge -1/(k-1)\}$.  Since $\Lnk
= \Ln \cap {\cal O}_{n,k}$ we have
$$N(\Lnk, X) = \{v + u ,|\, v \in N(\Ln,X), u \in N({\cal
  O}_{n,k},X)\}.$$ We consider $N(\Ln,X)$ and $N({\cal O}_{n,k},X)$
separately.  We show $J$ is not in the linear span of $N(\Ln,X)$ and
$J$ is orthogonal to $N({\cal O}_{n,k},X)$.  This implies $J$ is not
in the linear span of $N(\Lnk, X)$ and therefore $N(\Lnk, X)$ is not
full-dimensional.

For the sake of contradiction suppose $J$ is in the linear span of
$N(\Ln,X)$.  Then $J = R - M$ with $R$ and $M$ both in $N(\Ln,X)$.
This implies $M+J = R$ and $M+J \in N(\Ln,X)$.  Since $M \in N(\Ln,
X)$ we have $X = T(M)$ and by the lemma above $X$ is the Gram matrix
of $(v_1,\ldots,v_n)$ where,
$$v_i \propto \sum_{k \neq i} M_{i,k} v_k.$$
Since $M+J \in N(\Ln,X)$ we have $X=T(M+J)$ and the lemma also implies,
$$v_i \propto \sum_{k \neq i} (M_{i,k} + J_{i,k}) v_k
\Rightarrow v_i \propto v_j + \sum_{k \neq i} M_{i,k} v_k.$$
Therefore $$v_i \propto v_j.$$ Since $X_{i,j} \not\in \{-1,1\}$ the
unit vectors $v_i$ and $v_j$ can not be proportional and we have a
contradiction.  Therefore $J$ is not in the linear span of $N(\Ln,X)$.

Now note
$$N({\cal O}_{n,k},X) = \cone(\{e^{r,s} \,|\, X_{r,s} = -1/(k-1)\}),$$
where $e^{r,s}$ is the matrix that is 0 everywhere except in the
$(r,s)$ position which has value $1$. The dot product $J \cdot e^{r,s}$
equals $0$ for all $(r,s)$ with $X_{r,s} = -1/(k-1)$.  Therefore
$N({\cal O}_{n,k},X)$ is orthogonal to $J$.
\end{proof}

\section{Iterated linear optimization and rounding}
\label{sec:iter}

A key step in solving a combinatorial optimization problem via a
convex relaxation involves \emph{rounding} a solution of the convex
relaxation to a solution of the original combinatorial optimization
problem.  In the case of Max $k$-Cut this involves mapping a solution
$X \in \Lnk$ to a $k$-partition matrix $Y \in \Qnk$.

Recall that $X \in \Lnk$ is the Gram matrix of $n$ unit vectors 
$(v_1,\ldots,v_n)$.  
The problem of rounding $X$ can be seen
as a new clustering problem, where we would like to partition
$(v_1,\ldots,v_n)$ into $k$ sets.  To solve this clustering problem we look for
$Y \in \Qnk$ that is close to $X$.  Relaxing this 
problem to $\Lnk$ we obtain a new SDP with solution $X' \in \Lnk$.
Our rounding method involves repeating this process multiple times.
The process is guaranteed to converge and the underlying unit vectors
become more clearly clustered with each step.

\subsection{Iterated linear optimization}

Let $\Delta \subset \reals^n$ be a compact convex subset containing
the origin.  Let $T(x)$ be the map defined by linear optimization over
$\Delta$,
$$T(x) = \argmax_{y \in \Delta} x \cdot y.$$ In \cite{iter} it was
shown that 
fixed point iteration with $T$ always converges to a fixed
point of $T$.  Furthermore, when $\Delta$ is the elliptope, $T(X)$
solves a relaxation to the closest vertex problem.  Here we derive a
similar result for the $k$-way elliptope, and show that iterated
linear optimization in $\Lnk$ can be used to round a solution to the
SDP relaxation for Max $k$-Cut.

\subsection{Deterministic rounding in $\Lnk$}

Let $X \in \Lnk$ be a solution to the SDP relaxation of a Max
$k$-Cut problem.  If $X \in \Qnk$ then $X$ defines $k$-partition.
Otherwise we look for
$Y \in \Qnk$ that is closest to $X$.  By relaxing this 
problem we obtain a new SDP.  Solving
the new SDP leads to a new solution $Y \in \Lnk$.  If $Y \in \Qnk$
then $Y$ is the closest $k$-partition matrix to $X$.  Otherwise we
recursively look for a matrix $Z \in \Qnk$ that is closest to $Y$.
The approach leads to a fixed point iteration process with a map $T'$
that optimizes a linear function over $\Lnk$.

Let $a = (1-k/2)/(k-1)$ and $A$ 
be the matrix where every entry equals $a$.  If $Y \in
\Qnk$ then all entries in $Y$ are in $\{-1/(k-1),1\}$ and all entries in $Y+A$
are in
$$\left\{-\frac{k/2}{k-1}, \frac{k/2}{k-1}\right\}.$$
Therefore $(Y+A) \cdot (Y+A)$ is constant.

Let $Y \in \Qnk$ and consider the following expansion,
\begin{align*}
||X-Y||^2 & = ||(X+A)-(Y+A)||^2 \\
& = (X+A) \cdot (X+A) + (Y+A) \cdot
(Y+A) - 2(X+A) \cdot (Y+A).
\end{align*}
Note that $(X+A) \cdot (X+A)$ does not depend on $Y$ and $(Y+A) \cdot (Y+A)$ is constant.
Therefore the closest $k$-partition matrix to $X$ is,
\begin{align*}
Y = \argmin_{Y \in \Qnk} ||X-Y||^2 
& = \argmax_{Y \in \Qnk}\, (X + A) \cdot (Y + A), \\
& = \argmax_{Y \in \Qnk}\, (X + A) \cdot Y.
\end{align*}
Relaxing this problem to $\Lnk$ we obtain $Y = T(X+A)$ where $T$ is
defined over $\Delta = \Lnk$.

Let $T'(X) = T(X+A)$.  That is,
$$T'(X) = \argmax_{Y \in \Lnk}\, (X + A) \cdot Y.$$

\paragraph{Fixed point iteration} 
Our rounding method involves fixed point iteration with $T'$.  That
is, we generate a sequence $\{X_t\}$ where $X_0 = X$ and,
$$X_{t+1} = T'(X_t).$$

Note that iteration with $T'$ is equivalent to iteration with $T$ in
$\Delta = \Lnk+A$.

Figure~\ref{fig:iter} shows the result of the fixed point iteration
process in $\Lnk$ for different values of $k$.  In each case we start
from the result of the SDP relaxation of Max $k$-Cut for a graph with
$n=50$ vertices and random weights.  In each example fixed point
iteration with $T'$ converges to a $k$-partition matrix after a small
number of iterations.

Figure~\ref{fig:example} in
Section~\ref{sec:intro} shows an example of the sequence of solutions
generated by $T'$ for a geometric clustering problem.

We say that a fixed point $x$ of a map $f:\Delta \rightarrow \Delta$
is \emph{attractive} if $\exists \epsilon > 0$ such that $||x-x_0|| <
\epsilon$ implies that iteration with $f$ starting at $x_0$ converges
to $x$.  

In \cite{iter} we characterized all of the fixed points of $T'$ when $k=2$ (in this case $T'=T$) and showed the attractive fixed points are exactly the $k$-partition matrices.  Here we consider the case when $k \ge 2$.

\begin{proposition}
The $k$-partition matrices are attractive fixed points of $T'$.
\end{proposition}
\begin{proof}
  Let $X \in \Qnk$ be a $k$-partition matrix.

  Let $Y \in \Lnk$ with $||Y-X|| < (k/2)/(k-1)$.  Then
  $|(Y+A)_{i,j}-(X+A)_{i,j}| = |(Y-X)_{i,j}| < (k/2)/(k-1)$.  Since
  every entry in $X+A$ is either $-(k/2)/(k-1)$ or $(k/2)/(k-1)$ we
  see that $Y+A$ has the same sign pattern of $X+A$.

  For $Z \in \Lnk$ we have $|(Z+A)_{i,j}| \le |(X+A)_{i,j}|$.
  Therefore $(Y+A) \cdot (Z+A) \le (Y+A) \cdot (X+A)$.  Moreover if $Z
  \neq X$ we have $(Y+A) \cdot (Z+A) < (Y+A) \cdot (X+A)$.  This
  implies $X=T'(Y)$ and $X$ is an attractive fixed point of $T'$.
\end{proof}

\begin{figure}
    \centering
    \subfigure[$k=2$]{\includegraphics[trim=0 125 75 130, clip, width=4in]{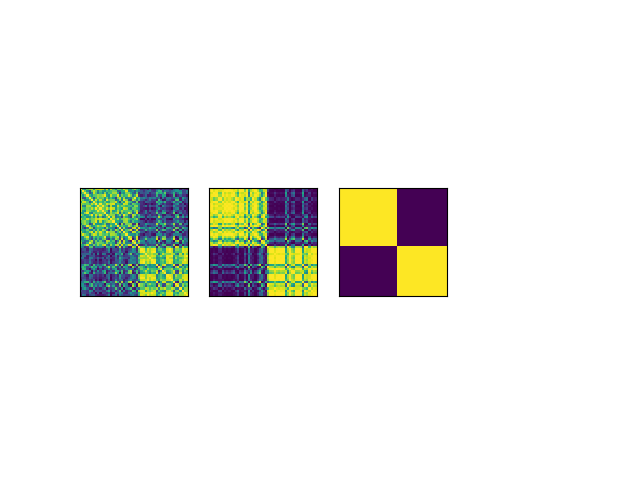}}    
    \subfigure[$k=5$]{\includegraphics[trim=20 140 65 130, clip, width=6in]{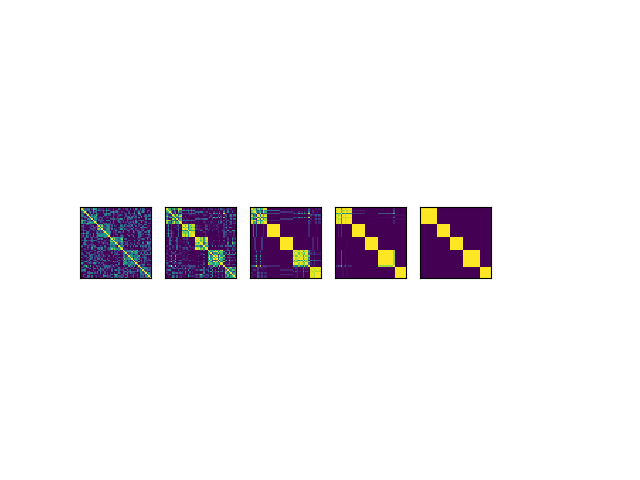}}    
    \subfigure[$k=10$]{\includegraphics[trim=20 140 65 130, clip, width=6in]{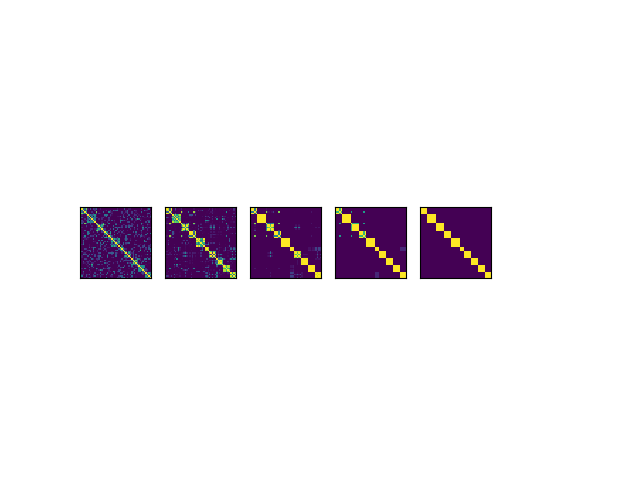}}
    \caption{Fixed point iteration with $T'$, starting from the
      solution of the SDP relaxation of Max $k$-Cut for a graph with
      50 vertices and random weights.  To visualize a matrix in $\Lnk$
      we show an $n \times n$ picture with a pixel for each entry in
      the matrix.  Bright yellow pixels indicate entries with high
      value, and dark blue pixels indicate entries with low value.  In
      each case we obtain a sequence of solutions that converge to a
      $k$-partition matrix.  The rows and columns in each example have
      been permuted so the final matrix is block diagonal to
      facilitate visualization.}
    \label{fig:iter}
\end{figure}

Define $f: \Lnk \rightarrow \mathbb{R}$ as,
$$f(X) = (X+A) \cdot (X+A) = \sum_{i=1}^n \sum_{j=1}^n (X_{i,j}+a)^2.$$  

Consider the representation of $X \in \Lnk$ as the Gram matrix of $n$
unit vectors $(v_1,\ldots,v_n)$ and recall that $X_{i,j} = v_i \cdot
v_j \ge -1/(k-1)$.

If $v_i \cdot v_j = 1$ (the vectors are as close as possible) then
$(X_{i,j}+a) = (k/2)/(k-1)$.  If $v_i \cdot v_j = -1/(k-1)$ (the
vectors are as far as possible) then $(X_{i,j}+a) =
-(k/2)/(k-1)$.  For any other choice $-(k/2)/(k-1) < (X_{i,j}+a) < (k/2)/(k-1)$.

Proposition~\ref{prop:entry} shows that $k$-partition matrices are
exactly the matrices in $\Lnk$ where every entry is in
$\{-1/(k-1),1\}$.  Therefore the matrices that maximize $f(X)$ are
the $k$-partition matrices.  We can see $f(X)$ as a measure of how
close $X$ is to a $k$-partition matrix.  We can also interpret $f(X)$
as a measure of how well clustered the vectors $(v_1,\ldots,v_n)$ are.

The convergence results from \cite{iter} imply the
sequence $\{X_t\}$ converges to a fixed point of $T'$.  
Additionally, the results from \cite{iter} imply that $f(X)$
increases in each iteration.  Therefore the vectors defining $X_t$ become
more clearly clustered in each iteration.

We have shown that matrices in $\Qnk$ are \emph{attractive} fixed
points of $T'$.  Although the map $T'$ has other fixed points, our
numerical experiments suggest that these are
the only attractive fixed points, and that fixed point iteration with $T'$
starting from a generic point in $\Lnk$ always converges to a
$k$-partition matrix.  

We have done several numerical experiments iterating $T'$ to round a solution of
the Max $k$-Cut relaxation.  In one setting we generated $50 \times
50$ random weight matrices $M$ with independent weights sampled from a
Gaussian distribution with mean 0 and standard deviation 1 (note that in this case the matrix $M$ may have negative weights).  In
another setting we generated $50 \times 50$ random weight matrices $M$
by sampling $50$ points in $\mathbb{R}^{10}$ independently from the
uniform distribution over $[0,1]^{10}$.  We then set $M_{ij} =
||x_i-x_j||^2$.  We ran 100 different trials with $k=5$ in each
setting.  Fixed point iteration with $T'$ starting from
the solution of the Max $k$-Cut relaxation always converged to a
$k$-partition matrix.  The number of iterations until convergence in
the first setting was between 3 and 10, with an average of 7.02.  The
number of iterations until convergence in the second setting was
between 3 and 4, with an average of 3.01.

\section{Clustering Experiments}
\label{sec:experiments}

In this section we illustrate the results of our clustering method on
several datasets.  The algorithms were implemented in Python and run
on a computer with an Intel i7 CPU @ 2.6 Ghz with 8GB of RAM.  We use
the cvxpy package for convex optimization together with the SCS
(splitting conic solver) package to solve SDPs.  The fixed point
iteration process we use for rounding involves solving a sequence of
SDPs identical to the Max $k$-Cut SDP but with a different objective.
For the examples below solving each SDP took 1 to 3 minutes.  The
fixed point iteration method converged after 1 to 5 iterations in each
case.

In each clustering experiment we have a dataset $D=\{x_1,\ldots,x_n\}$
with $x_i \in \reals^d$.  As discussed in
Section~\ref{sec:clustering} we cluster the data by defining a Max
$k$-Cut problem with weights $M_{i,j} = ||x_i-x_j||^2$.  With this
choice of weights the maximum weight partition should lead to compact and balanced 
clusters, where points within a single cluster are close together while points in different clusters are far from each other.  Section~\ref{sec:synt} illustrates
clustering results on synthetic data, while Section~\ref{sec:mnist} evaluates the results of
clustering a subset of the MNIST handwritten digits (\cite{MNIST}).

\subsection{Synthetic data}
\label{sec:synt}

We performed several experiments using synthetic data in $\reals^2$ to
facilitate the visualization of the results.
Figures~\ref{fig:test5},~\ref{fig:test4}, and~\ref{fig:test3}
illustrate the results of clustering a dataset of 200 points into 5,
10 and 20 clusters respectively.\footnote{We used a subset of the D-31
dataset from \cite{Veenman02} with 10 points from each of 20
clusters.}  In each case we compare the result obtained using fixed
point iteration for rounding the solution of the Max $k$-Cut
relaxation to the result of randomized rounding.  For the randomized
rounding method we repeat the rounding procedure 50 times and select
the partition with highest weight generated over all trials.

In all of the experiments we see that the weight of the partition
generated by the fixed point iteration method is higher than the
weight of the best partition generated by randomized rounding.  When
$k=5$ (Figure~\ref{fig:test5}) the results of the two methods are
similar, but the weight of the partition generated by fixed point
iteration is slightly better.  When $k=10$ and $k=20$
(Figures~\ref{fig:test4} and~\ref{fig:test3}) the results of
randomized rounding are significantly degraded while the results of
the fixed point iteration method remain very good.

Figure~\ref{fig:random} illustrates the partitions obtained in
different trials of randomized rounding for the case when $k=5$.  We
can see there is a lot of variance in the results and that the random
rounding method often leads to poor clusterings even when $k$ is
relatively small.  The result of our fixed point iteration method in
the same data is shown in Figure~\ref{fig:test5}(c).

\begin{figure}
\centering
  \subfigure[Input data]{\includegraphics[width=3in]{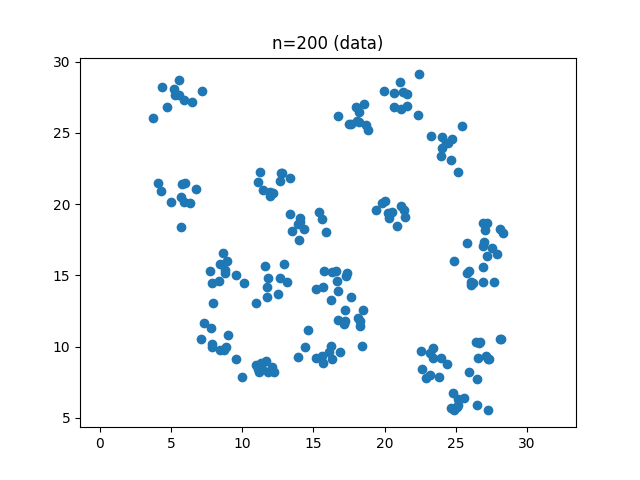}}
  
  \subfigure[Randomized (best of 50) $w(C) = 3543294$]{\includegraphics[width=3in]{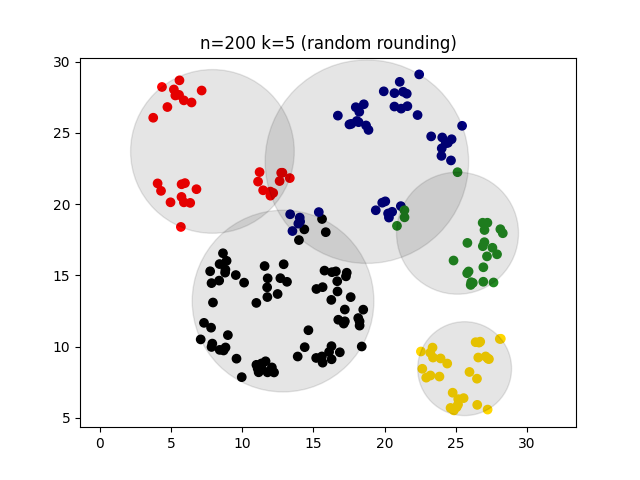}}
  \subfigure[Fixed point iteration $w(C) = 3589259$]{\includegraphics[width=3in]{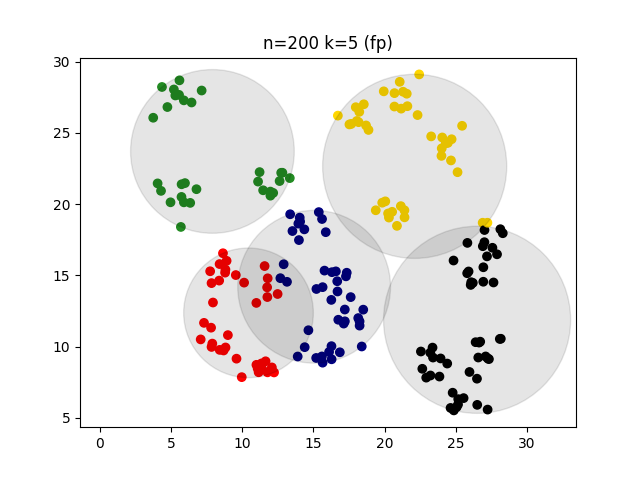}}
  \caption{Clustering 200 points with $k=5$.}
\label{fig:test5}
\end{figure}

\begin{figure}
\centering
  \subfigure[Input data]{\includegraphics[width=3in]{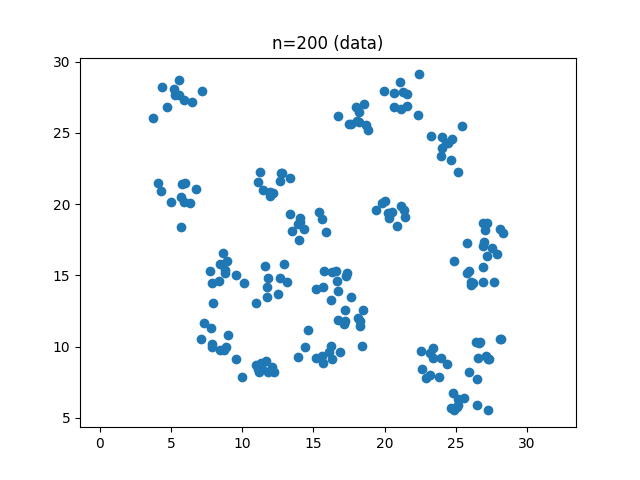}}
  
  \subfigure[Randomized (best of 50) $w(C) = 3587153$]{\includegraphics[width=3in]{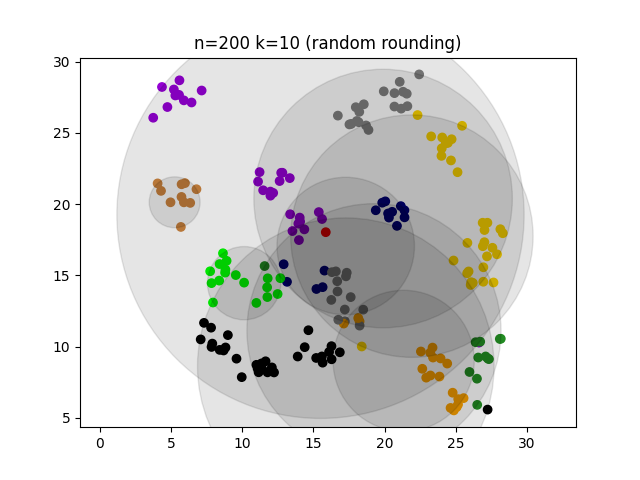}}
  \subfigure[Fixed point iteration $w(C) = 3701677$]{\includegraphics[width=3in]{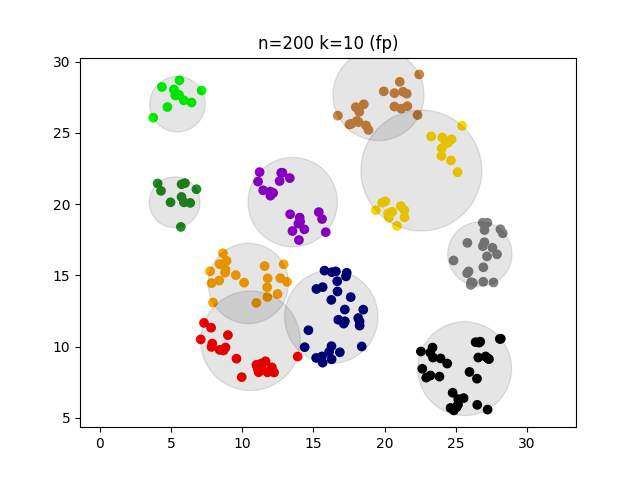}}
  \caption{Clustering 200 points with $k=10$.}
\label{fig:test4}
\end{figure}

\begin{figure}
\centering
  \subfigure[Input data]{\includegraphics[width=3in]{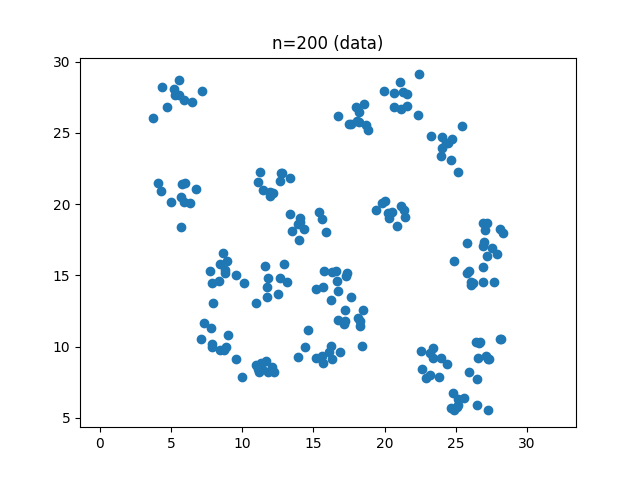}}
  
  \subfigure[Randomized (best of 50) $w(C) = 3658976$]{\includegraphics[width=3in]{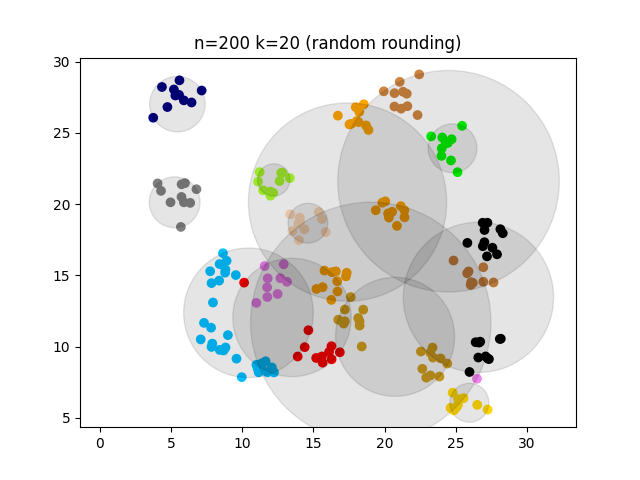}}
  \subfigure[Fixed point iteration $w(C) = 3722073$]{\includegraphics[width=3in]{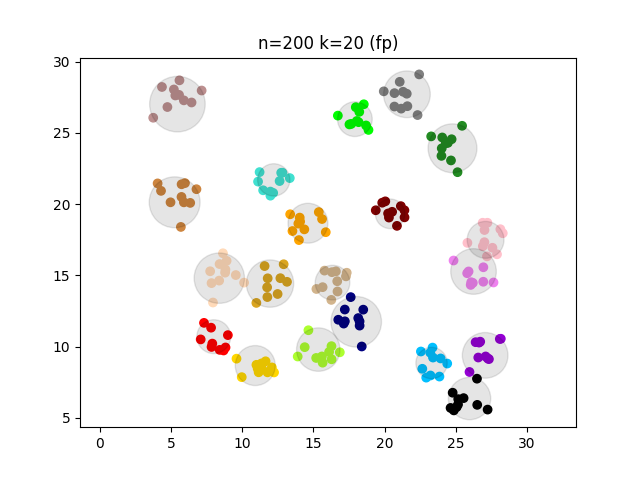}}
  \caption{Clustering 200 points with $k=20$.}
\label{fig:test3}
\end{figure}

\begin{figure}
\centering
  \subfigure[Trial 1]{\includegraphics[width=2.8in]{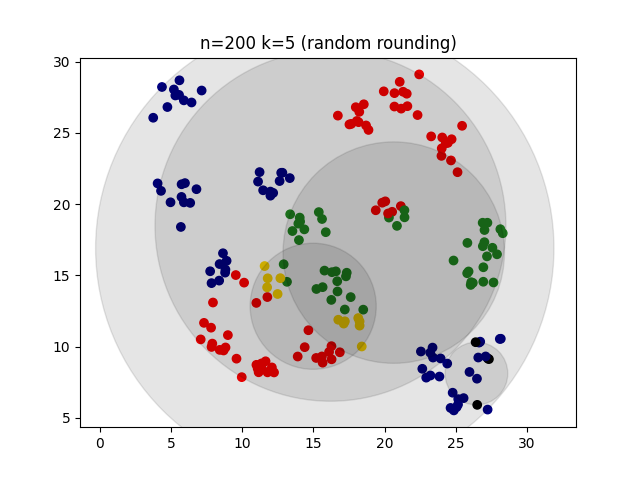}}
  \subfigure[Trial 2]{\includegraphics[width=2.8in]{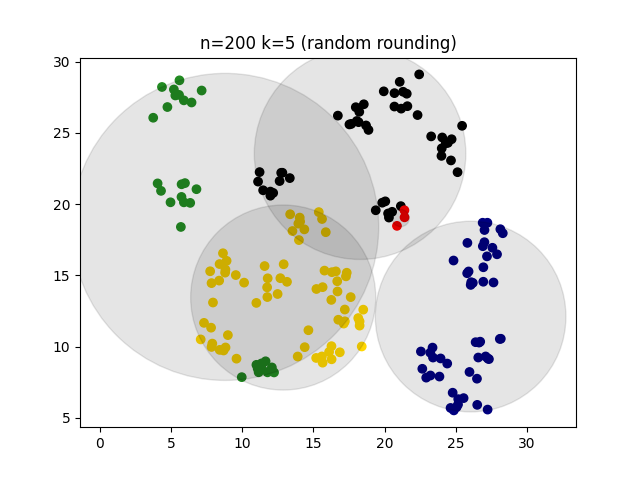}}
  
  \subfigure[Trial 3]{\includegraphics[width=2.8in]{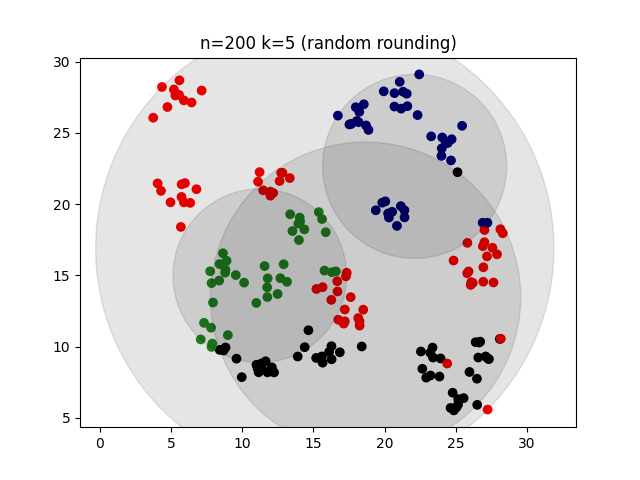}}
  \subfigure[Trial 4]{\includegraphics[width=2.8in]{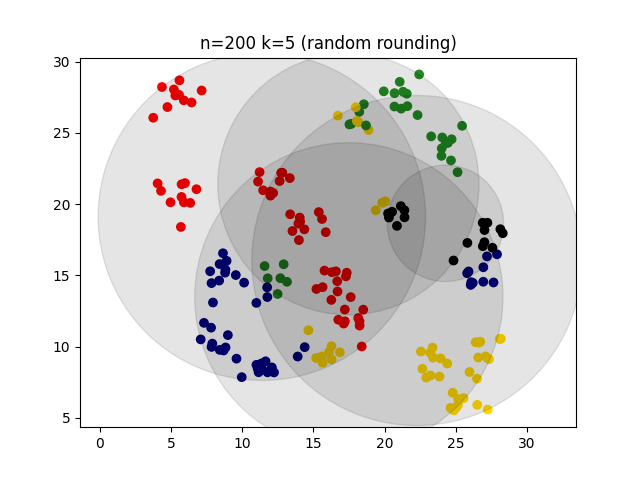}}
  
  \subfigure[Trial 5]{\includegraphics[width=2.8in]{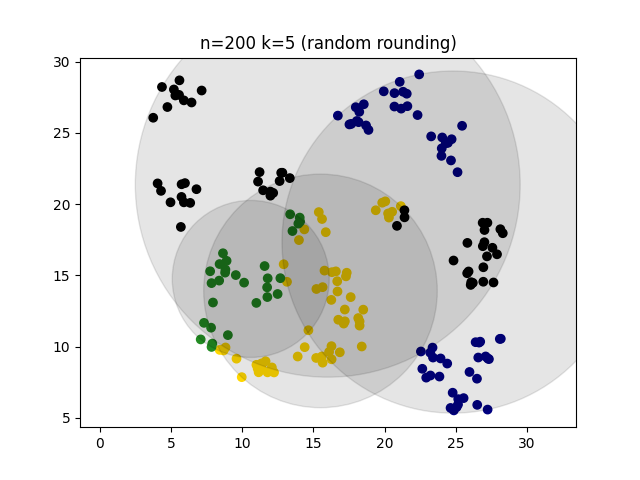}}
  \subfigure[Trial 6]{\includegraphics[width=2.8in]{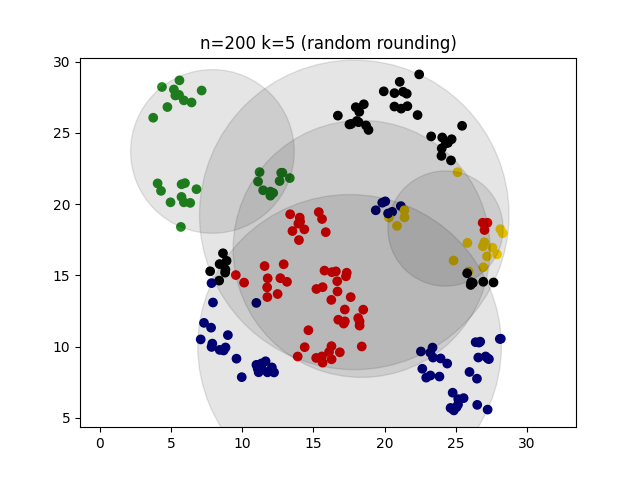}}  
  \caption{Clustering 200 points with $k=5$ and randomized rounding.  The result of the fixed point iteration method in the same data is shown in Figure~\ref{fig:test5}(c).}
\label{fig:random}
\end{figure}

Figure~\ref{fig:gaussian} shows an example where the input data has
160 points generated by sampling 20 points from 8 different Gaussian
distributions.  The Gaussian distributions have standard deviation
$\sigma=0.2$ and means arranged around a circle of radius 1.  In this
case there is a ground truth clustering where points are grouped
according to the Gaussian used to generate them.  The overlap between
the distributions is sufficiently high that it is impossible to
recover the ground truth clustering perfectly, but the result of the
fixed point iteration method is closely aligned with the ground truth.
The results of randomized rounding are not as good even when we select
the best of 50 trials of the randomized procedure.

We repeated the last experiment 10 times to quantify the difference
between the two rounding methods.  Each repetition was done with a new
dataset generated from the mixture of Gaussians.  
Table~\ref{tab:gaussians}(a) summarizes the minimum, maximum,
and mean value of the ratio $w(C)/w(D)$, where $C$ is the partition
produced by the fixed point iteration method and $D$ is the result of
the best of 50 trials of randomized rounding.  In these experiments 
the ratio $w(C)/W(D)$ was
always above one, showing that our fixed point iteration method 
always returns a partition with
higher weight.  We also compare the
resulting clusterings to the ground truth using the Rand index
(\cite{rand}).  

\begin{table}
\centering
\subtable[Partition weights]{\begin{tabular}{|l|c|}
\hline 
& $w(C)/w(D)$ \\
\hline
minimum & 1.005 \\
\hline
maximum & 1.026 \\
\hline
mean & 1.014 \\
\hline
\end{tabular}}
\subtable[Clustering accuracy]{\begin{tabular}{|l|c|}
\hline 
& Rand index \\
\hline 
fixed point & $0.972 \pm 0.006$ \\
\hline 
randomized rounding & $0.935 \pm 0.018$ \\
\hline
\end{tabular}}
\caption{Experiments with 10 different datasets generated by a mixture
  of Gaussians (see Figure~\ref{fig:gaussian}). (a) Minimum, maximum
  and mean value of $w(C)/w(D)$, where $C$ is the result of the fixed
  point iteration method and $D$ is the result of the best of 50
  trials of randomized rounding. (b) Mean and standard deviation of
  the Rand index for both methods.}
\label{tab:gaussians}
\end{table}

\begin{figure}
\centering
  \subfigure[Input data and ground truth clustering]{\includegraphics[width=3in]{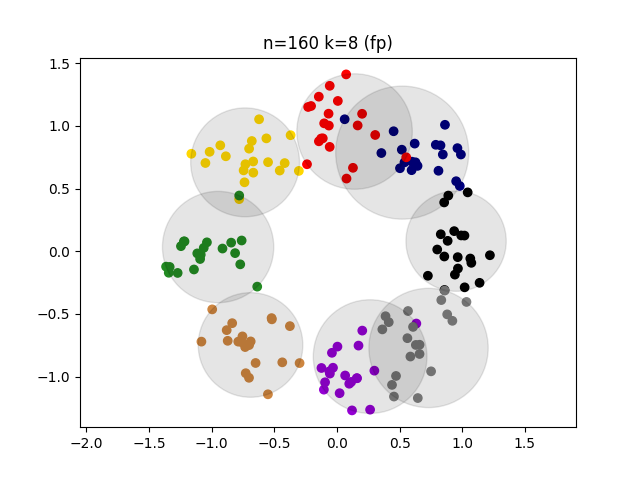}}
  
  \subfigure[Randomized (best of 50) $w(C) = 26837$]{\includegraphics[width=3in]{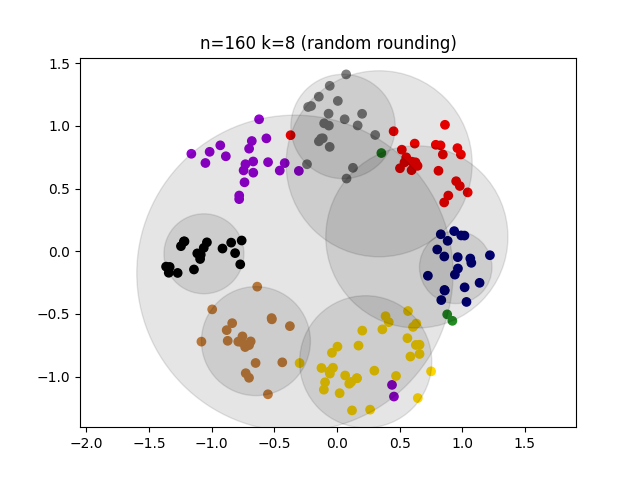}}
  \subfigure[Fixed point iteration $w(C) = 27158$]{\includegraphics[width=3in]{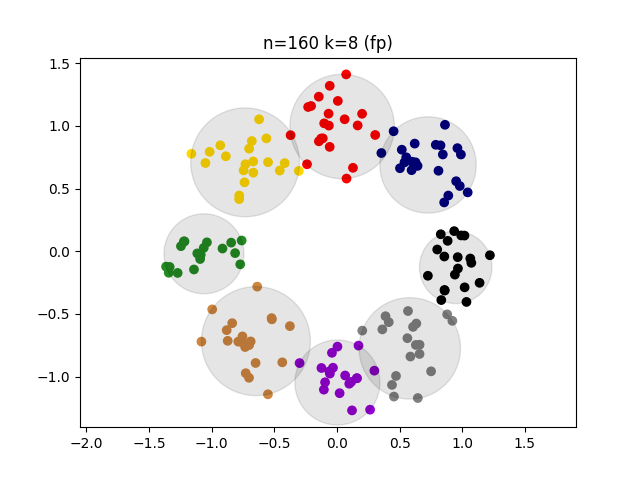}}
  \caption{Clustering 160 points sampled from 8 Gaussians  with $k=8$.}
\label{fig:gaussian}
\end{figure}

The Rand index is a number between 0 and 1 measuring
the agreement between two clusterings.  Let $A$ and $B$ be
clusterings.  Let $a$ be the number of pairs of elements that are in
the same cluster in both $A$ and $B$ while $b$ is the number of pairs
of elements that are in different clusters in both $A$ and $B$.  The
Rand index is,
$$R(A,B) = \frac{a+b}{\binom{n}{2}}.$$
We see in Table~\ref{tab:gaussians}(b) that our fixed point
rounding method consistently produces a clustering that
is more similar to the ground truth than the randomized rounding method.

\subsection{MNIST digits}
\label{sec:mnist}

In this section we evaluate the performance of our clustering method
on a subset of the MNIST handwritten digits dataset (\cite{MNIST}).

We performed multiple clustering trials with different subsets of the
MNIST training data.  In each trial we selected 20 random examples for
each of 5 digits (0, 1, 2, 3 and 4).  Each example was represented by
a 784 dimensional $0/1$ vector encoding a 28x28 binary image.
Figure~\ref{fig:mnist} shows the data from one of the trials.

\begin{figure}
    \centering
    \includegraphics[width=1.5in]{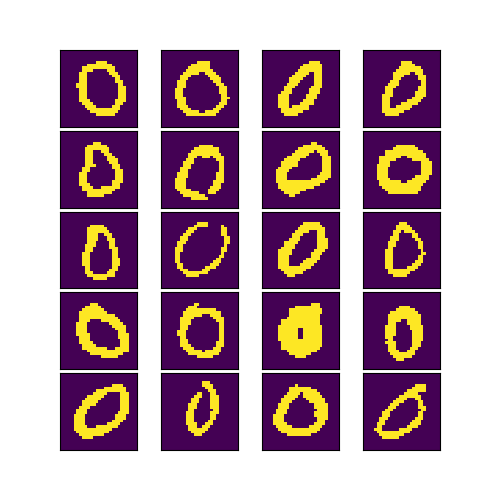}
    \includegraphics[width=1.5in]{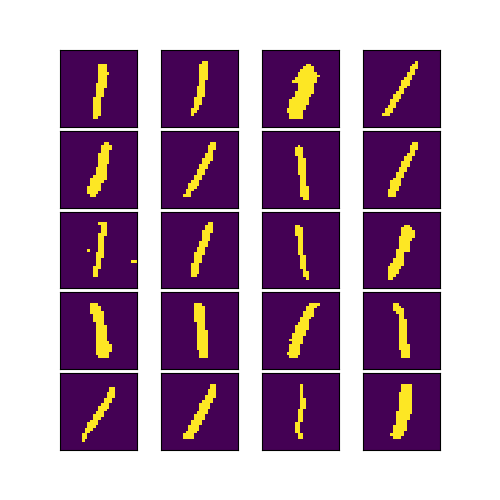}
    \includegraphics[width=1.5in]{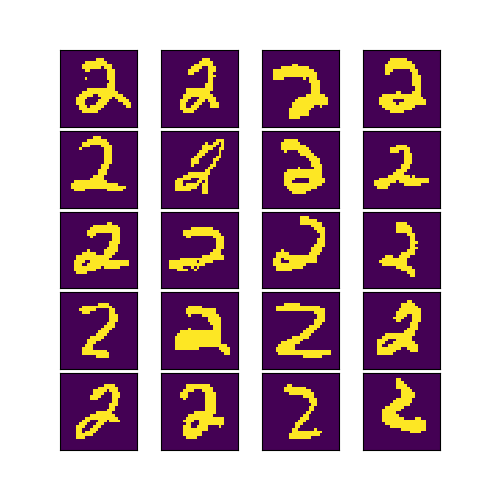} 
    
    \vspace{-.2cm}
    \includegraphics[width=1.5in]{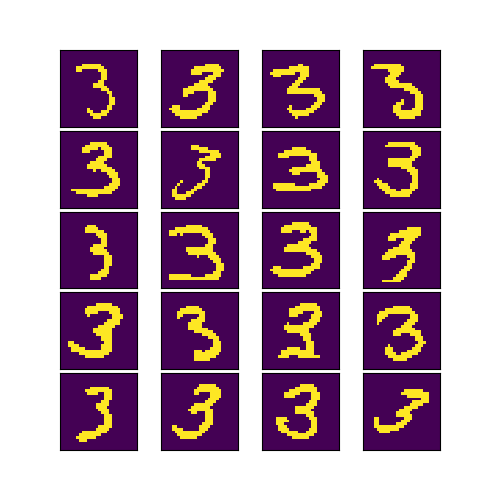}
    \includegraphics[width=1.5in]{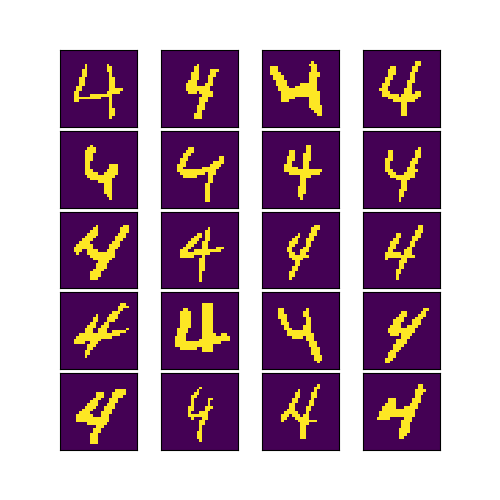}
    \vspace{-.2cm}
    \caption{Subset of MNIST handwritten digits used in one of the clustering trials.}
    \label{fig:mnist}
\end{figure}

In Table~\ref{tab:mnist} we report the accuracy of the clustering
results obtained using both our fixed point iteration method and the
randomized method for rounding the Max $k$-Cut relaxation.  For
comparison we also evaluate the results of clustering the data using
the $k$-means algorithm.  We see that our fixed point iteration method
for rounding the Max $k$-Cut relaxation gives the best accuracy when
compared both to the use of randomized rounding and the $k$-means
algorithm.\footnote{We used the scipy library implementation of
$k$-means.  Each trial involved 10 different random initializations of
the initial cluster centers using the $k$-means++ method.}

\begin{table}
\centering
\begin{tabular}{|l|c|}
\hline 
Method & Rand index \\
\hline 
Max $k$-Cut SDP with fixed point iteration & $0.907 \pm 0.026$ \\
\hline 
Max $k$-Cut SDP with randomized rounding & $0.874 \pm 0.037$ \\
\hline
k-means & $0.864 \pm 0.034$ \\
\hline
\end{tabular}
\caption{Clustering MNIST hardwritten digits.  We report both the mean
  and standard deviation of the Rand index over 20 trials with
  random subsets of the data.}
\label{tab:mnist}
\end{table}

\bibliography{biblio.bib}

\end{document}